\documentclass[12pt]{amsart} 
\usepackage{tikz-cd}
\usepackage{fancyhdr} 
  
\usepackage{latexsym, amssymb, amscd, amsfonts,pstricks,enumitem,amsmath, young}
\usepackage{latexsym, amssymb, amscd, amsfonts,pstricks,enumitem,amsmath}
\usepackage{amsmath,amsthm,amssymb,mathrsfs}
 
\usepackage[all]{xypic}

\renewcommand{\geq}{\geqslant}
\renewcommand{\leq}{\leqslant}
\newcommand{\Osh}{{\mathcal O}}                        %  Structure sheaf

\newcommand{\K}{\mathrm{K}}                            
\newcommand{\kk}{\mathbf{k}}

\newcommand{\KK}{\mathbf{K}}
\newcommand{\FF}{\mathbf{F}}
\newcommand{\PP}{\mathbb{P}} % projective space
\newcommand{\QQ}{\mathbb{Q}} % rational numbers
\newcommand{\RR}{\mathbb{R}} % real numbers
\newcommand{\ZZ}{\mathbb{Z}} % integers

\newtheorem{theorem}{Theorem}[section]

\theoremstyle{definition}

\newtheorem{defns}[theorem]{Definitions}
\newtheorem{remark}[theorem]{Remark}

\newtheorem{example}[theorem]{Example}

\numberwithin{equation}{section}

\begin{document}

\title[Harder-Narasimhan polygons and Laws of Large Numbers]{Vertices of the Harder and Narasimhan polygons and the Laws of Large Numbers
}

\author{Nathan Grieve}
\address{Department of Mathematics \& Computer Science,
Royal Military College of Canada, P.O. Box 17000,
Station Forces, Kingston, ON, K7K 7B4, Canada
}
\address{School of Mathematics and Statistics, 4302 Herzberg Laboratories, Carleton University, 1125 Colonel By Drive, Ottawa, ON, K1S 5B6, Canada
}
\address{
D\'{e}partement de math\'{e}matiques, Universit\'{e} du Qu\'{e}bec \`a Montr\'{e}al, Local PK-5151, 201 Avenue du Pr\'{e}sident-Kennedy, Montr\'{e}al, QC, H2X 3Y7, Canada}
\email{nathan.m.grieve@gmail.com}%

\begin{abstract} 
We build on the recent techniques of Codogni and Patakfalvi, from \cite{Codogni:Patakfalvi:2021}, which were used to establish theorems about semi-positivity of the Chow Mumford line bundles for families of $\K$-semistable Fano varieties.  Here we apply the Central Limit Theorem to ascertain the asymptotic probabilistic nature of the vertices of the \emph{Harder and  Narasimhan polygons}.  
As an application of our main result, we use it to establish a filtered vector space analogue of the main technical result of \cite{Codogni:Patakfalvi:2021}.  In doing so, we expand upon the slope stability theory, for filtered vector spaces, that was initiated by Faltings and W\"{u}stholz \cite{Faltings:Wustholz}.  One source of inspiration for our abstract study of \emph{Harder and  Narasimhan data}, which is a concept that we define here,
is the lattice reduction methods of Grayson \cite{Grayson:1984}.  Another is the work of Faltings and W\"{u}stholz, \cite{Faltings:Wustholz}, and Evertse and Ferretti, \cite{Evertse:Ferretti:2013}, which is within the context of Diophantine approximation for projective varieties.
\end{abstract}

\thanks{
\emph{Mathematics Subject Classification (2020):} 14F06, 11J68, 14J10. \\
\emph{Key Words:} 
Harder and Narasimhan filtrations, Diophantine approximation, $\K$-stability. \\
I thank the Natural Sciences and Engineering Research Council of Canada for their support through my grants DGECR-2021-00218 and RGPIN-2021-03821. \\
}

\maketitle

\section{Introduction}

Our purpose here, is  to continue our work which is at the intersection of $\K$-stability and Diophantine approximation for projective varieties (\cite{Grieve:Function:Fields},  \cite{Grieve:2018:autissier},  \cite{Grieve:Divisorial:Instab:Vojta}, \cite{Grieve:toric:gcd:2019}, \cite{Grieve:points:bounded:degree}, \cite{Grieve:MVT:2019} and  \cite{Grieve:chow:approx}).  
In more specific terms, we expand upon the theory Harder and Narasimhan filtrations for holomorphic vector bundles on compact Riemann surfaces.  Especially, we define and study asymptotic probabilistic features of \emph{Harder and Narasimhan data}. (See Section \ref{HN:polygons:and:vertices} and Theorems \ref{main:theorem:intro} and \ref{filt:HN:tensor:cor}.)

Recall, that Harder and Narasimhan's theory is now classical.  It was conceived in \cite{Harder:Narasimhan:1974}.  Recently, a significant application of this theory, building on earlier work of Viehweg, has been given by Codogni and Patafalvi \cite{Codogni:Patakfalvi:2021}.  Among other results, in \cite{Codogni:Patakfalvi:2021}, the theory of Harder and Narasimhan filtrations is applied to obtain results for the Chow-Mumford line bundles that arise within the context of families of $\K$-stable klt Fano varieties.  

For example, the concept of Harder and Narasimhan filtration is a key, more technical, tool, which is used in \cite{Codogni:Patakfalvi:2021}, to establish the following result.

\begin{theorem}[{\cite[Theorem 1.1 (a)]{Codogni:Patakfalvi:2021}}]
Fixing an integer $n>0$ and a rational number $v>0$, let $\mathcal{M}^{\K-\mathrm{ss}}_{n,v}$ be the moduli stack of those $\K$-semistable dimension $n$ Fano varieties, which have anti-canonical volume equal to $v$.  Over $\mathcal{M}^{\K-\mathrm{ss}}_{n,v}$, let $\lambda$ be the Chow-Mumford line bundle.  Then $\lambda$ and its descent along the good moduli space morphism 
$$\mathcal{M}^{\K-\mathrm{ss}}_{n,v} \rightarrow \mathrm{M}_{n,v}^{\K,\mathrm{ps}}$$ 
are numerically effective.
\end{theorem}

In terms of Harder and Narasimhan filtrations and Diophantine approximation, the main result is the theorem of Faltings and W\"{u}stholz, from \cite{Faltings:Wustholz}, and the refinement which was established by Evertse and Ferretti in \cite{Evertse:Ferretti:2013}.  Here, in Theorem \ref{FW:theorem} below, we state a version of that result.  McKinnon and Roth formulated a corresponding statement for linear systems \cite{McKinnon-Roth}.    

Investigations and expansions of Schmidt's Subspace Theorem, from the viewpoint of linear systems, is another topic of continued recent and ongoing interest (\cite{Ru:Wang:2016}, \cite{Ru:Vojta:2016}, \cite{Heier:Levin:2017}, \cite{Grieve:2018:autissier}, \cite{Grieve:points:bounded:degree}).  Finally, recall, as explained in \cite{Evertse:Ferretti:2013},  that Schmidt's Subspace Theorem can be deduced from the inequalities of Faltings and W\"{u}stholz.  Conversely, Schmidt's Subspace Theorem can be used to imply the inequalities of Faltings and W\"{u}stholz.  In particular, the following form of the celebrated theorem of Faltings and W\"{u}stholz is stated in \cite[p.~ 514]{Evertse:Ferretti:2013}.

\begin{theorem}[\cite{Faltings:Wustholz}, \cite{McKinnon-Roth}, \cite{Evertse:Ferretti:2013}]\label{FW:theorem}
Let $\KK$ be a number field $M_{\KK}$ its set of places and 
$S \subseteq M_{\KK}$ 
a finite subset.  For  each 
$v \in M_{\KK}$ 
fix normalized absolute values $|\cdot|_v$ so that the product theorem holds true with multiplicities equal to one.  For each $v \in S$, fix linearly independent linear forms 
$$
\ell_{0v}(x),\dots,\ell_{nv}(x) \in \KK[x_0,\dots,x_n] \text{.}
$$
together with nonnegative weights $d_{iv} \in \RR_{\geq 0}$ which have the property that
$$
\sum_{v \in  S} \sum_{i=0}^n d_{iv} >  n + 1 \text{.}
$$
Let $H_{\Osh_{\PP^n_{\KK}}(1)}(\cdot)$ be the multiplicative height function on projective $n$-space $\PP^n_{\KK}$ with respect to the tautological line bundle $\Osh_{\PP^n_{\KK}}(1)$.  Let $\epsilon > 0$.  Then, with this notation and hypothesis, there exists a single, effectively computable proper linear subspace 
$$T \subsetneq \PP^n(\KK)$$ 
such that system
$$
\frac{| \ell_{iv}(x) |_v }{ \max_{0 \leq i \leq n } |x_i|_v} \leq H_{\Osh_{\PP^n}(1)}(x)^{-d_{iv}} \text{,}
$$
for $v \in S$ and $i = 0,\dots,n$,
admits at most a finite number of solutions 
$$
x \in \PP^n(\KK) \setminus T \text{.}
$$
\end{theorem}

Finally, we mention our recent result, from \cite{Grieve:Divisorial:Instab:Vojta}, which is at the intersection of $\K$-stability and Diophantine approximation. It establishes, in particular, that the concept of $\K$-instability for Fano varieties has implications for instances of Vojta's Main Conjecture.

\begin{theorem}[{\cite[Theorem 1.1]{Grieve:Divisorial:Instab:Vojta}}]\label{Vojta:canonical:Fano:not:K:stable}
Let $\KK$ be a number field and fix a finite set of places $S$ of $\KK$.  Suppose that $X$ is a $\QQ$-Fano variety with canonical singularities, defined over $\KK$, and which is not $\K$-stable.  Then over $X$ there exists a nonzero, irreducible and reduced effective Cartier divisor $E$, which is defined over some finite extension field $\FF / \KK$, with $\KK \subseteq \FF \subseteq \overline{\KK}$, for which the inequalities predicted by Vojta's Main Conjecture hold true in the following sense.  Let $\mathbb{E}$ be the birational divisor that is determined by $E$.  Let 
$$\mathbb{D} = \mathbb{D}_1+\dots+\mathbb{D}_q$$ 
be a birational divisor over $X$ that has the two properties that: 
\begin{enumerate}
\item[(i)]{
the traces of each of the $\mathbb{D}_i$ are linearly equivalent to the trace of $\mathbb{E}$ on some fixed normal proper model $X'$ of $X$, defined over $\FF$; and
}
\item[(ii)]{
the traces of each of these divisors $\mathbb{D}_i$, for $i = 1,\dots,q$, intersect properly on this model $X'$.
}
\end{enumerate}
Let $B$ be a big line bundle on $X$ and let $\epsilon > 0$.  Then the inequality
\begin{equation}\label{Vojta:Inequality:Eqn:intro:b}
\sum_{v \in S} \lambda_{\mathbb{D},v}(x) + h_{\K_X}(x) \leq \epsilon h_B(x) + \mathrm{O}(1)
\end{equation}
is valid for all $\KK$-rational points $x \in X(\KK) \setminus Z(\KK)$ and $Z \subsetneq X$ some proper Zariski closed subset defined over $\KK$.
\end{theorem}

In \eqref{Vojta:Inequality:Eqn:intro:b}, $\lambda_{\mathbb{D},v}(\cdot)$ is the \emph{birational Weil function} of $\mathbb{D}$ with respect to the place $v \in S$.  (We refer to \cite[Section 4]{Ru:Vojta:2016} and \cite[Section 3]{Grieve:Divisorial:Instab:Vojta} for more details.)
 
Here, our main result builds on the techniques from \cite{Codogni:Patakfalvi:2021} and establishes, in particular, a Central Limit Theorem for the vertices of the Harder-Narasimhan polygons.  (See Theorem \ref{main:theorem:intro} and Sections \ref{HN:polygons:and:vertices} and \ref{Main:Theorem:Proof} for precise details.)

To place matters into perspective, note that an essential feature to the work of Faltings and W\"{u}stholz, \cite{Faltings:Wustholz}, is their theory of \emph{slope semi-stability for filtered vector spaces}.  Over the years, this theory has been developed and has produced significant applications.  (See, for example, \cite{Faltings:1994:ICM}, \cite{Totaro:Tensor:prod}, \cite{Fujimori:2003}, \cite{Chen:2010:b} and the references therein.)  

Our main results here,  Theorems \ref{main:theorem:intro} and \ref{filt:HN:tensor:cor} below, are natural extensions to this circle of ideas.  Aside from being of an intrinsic interest in their own right, they provide continued evidence for the existence of fruitful, and yet to be discovered, interactions amongst the areas of K-stability, positivity questions for polarized projective varieties and Diophantine arithmetic geometry.

Let us now formulate Theorem \ref{main:theorem:intro}.  In Section \ref{Main:Theorem:Proof}, see Theorem \ref{HN:vertex:CLT}, we state a slight variant, which is phrased in terms of our concept of \emph{Harder and Narasimhan data}.  In our formulation of Theorem \ref{main:theorem:intro}, we define this concept in passing.

\begin{theorem}\label{main:theorem:intro}
Fix a \emph{rank vector} 
$\overrightarrow{r} := (r_1,\dots,r_\ell)\text{,}$
that consists of positive integers $r_i$, for $i = 1,\dots,\ell$, 
fix a collection of strictly decreasing collection of rational numbers
$\mu_1 > \dots > \mu_\ell$ 
and let 
$\overrightarrow{d} := (d_1,\dots,d_\ell)$ 
be the resulting \emph{degree vector}, which is determined by the condition that
$d_i := \mu_i r_i \text{,}$
for $i = 1,\dots,\ell$.      Assume that such \emph{Harder and Narasimhan data} is  \emph{positive} in the sense that 
$$\left|\overrightarrow{d}\right| := \sum_{i=1}^\ell d_i > 0 \text{.} $$  
Let 
$p_i := r_i / r\text{,}$
for $i = 1,\dots,\ell$, and let 
$\mathcal{Y} := \mathcal{Y}\left([\ell],p_1,\dots,p_{\ell}\right)$ 
be the discrete probability  space on the set 
$[\ell] := \{1,\dots,\ell \}$ 
and having probability measures $p_1,\dots,p_\ell$.  Let 
$Y_j := Y_j(\overrightarrow{\mu})$ 
be a sequence of independent, identically distributed random variables of $\mathcal{Y}$ that take value 
$\mu_i := d_i / r_i$ 
on $i$.  
 Then, within this context, given a nonnegative integer $z \geq 0$, it holds true that
$$\lim_{m \to \infty} \operatorname{Prob}\left( \sum_{j=1}^m Y_j \geq z \right) = 1 \text{.}$$
\end{theorem}

We prove Theorem \ref{main:theorem:intro} in Section \ref{Main:Theorem:Proof}.  
As one application, it implies a filtered vector space analogue of \cite[Theorem 5.11]{Codogni:Patakfalvi:2021}.  This is the content of Theorem \ref{filt:HN:tensor:cor}, whose formulation is more technical.  

To get a flavour, consider a \emph{filtered vector space}
\begin{equation}\label{filt:vp:space:eqn1}
V = F^{\lambda_0}V \supsetneq F^{\lambda_1}V \supsetneq \dots \supsetneq F^{\lambda_n}V \supsetneq F^{\lambda_{n+1}} V = 0 \text{.}
\end{equation}
Then, by the theory of Faltings and W\"{u}stholz, \cite[Section 4]{Faltings:Wustholz},  
it admits a \emph{canonical Harder and Narasimhan filtration}
\begin{equation}\label{filt:vp:space:eqn2}
0 = V_0 \subsetneq V_1 \subsetneq \dots \subsetneq V_\ell = V \text{.}
\end{equation}
  The \emph{Harder and Narasimhan data}, $\operatorname{HN}(\overrightarrow{\mu},\overrightarrow{r})$, is obtained from the Harder and Narasimhan filtration \eqref{filt:vp:space:eqn2}.  It consists  of the \emph{slope} and \emph{rank} vectors.  We refer to Example \ref{filt:v:sp:example} for further details.

To get a sense for what we mean by \emph{positive} Harder and Narasimhan data,
respectively, denote the \emph{slope}, \emph{rank} and \emph{degree vectors} as
$$
\overrightarrow{\mu} := (\mu_1,\dots,\mu_\ell) \text{, } \overrightarrow{r} := (r_1,\dots,r_\ell) \text{ and } \overrightarrow{d} := (d_1,\dots,d_\ell) \text{.}
$$
Here
$$
\mu_i := \mu(V_i / V_{i-1}) \text{, } 
r_i := \dim (V_i/V_{i-1}) 
\text{ and }
d_i := r_i \mu_i \text{,}
$$
for $i = 1,\dots, \ell$.  The Harder and Narasimhan data $\operatorname{HN}(\overrightarrow{\mu},\overrightarrow{r})$ is called \emph{positive} if
$$
\left|\overrightarrow{d}\right| := \sum_{i=1}^{\ell} d_i > 0 \text{.}
$$

The intuition for Theorem \ref{filt:HN:tensor:cor}, is that, fixing $m \geq 0$, we want to study, inside of $V^{\otimes m}$, the collection of those subspaces that are spanned by $m$-fold tensor products of the \emph{Harder and Narasimhan subspaces} $V_i$, which appear in the Harder and Narasimhan filtration \eqref{filt:vp:space:eqn2}.

In more detail, for each 
$$
\mathbf{a} = (a_1,\dots,a_m) \in [\ell]^m := \{1,\dots,\ell\}^m \text{,}
$$
set
$$
v_{\overrightarrow{\mu}}(\mathbf{a}) := \sum_{j=1}^m \mu_{a_j} \text{,}
$$
and fixing a nonnegative integer $z \geq 0$, set 
$$
S_{m,z} (\overrightarrow{\mu}):= \left\{ \mathbf{a} = (a_1,\dots,a_m) \in [\ell]^m : \sum_{j=1}^m \mu_{a_j} \geq z \right\} \text{.}
$$
Finally, writing the elements of $[\ell]^m$ in decreasing order with respect to $v_{\overrightarrow{\mu}}(\cdot)$, denoted as
$$
[\ell]^m := \{\mathbf{a}_1 = (a_{11},\dots,a_{m1}),\dots,\mathbf{a}_{\ell m} = (a_{1 \ell m},\dots,a_{m \ell m}) \}\text{,}
$$
put 
\begin{equation}\label{HN:subspace}
H^{\# S_{m,z}(\overrightarrow{\mu})} := \sum_{i=1}^{\# S_{m,z}(\overrightarrow{\mu})} \bigotimes_{j=1}^m V_{a_{ji}} \text{.}
\end{equation}

Having fixed some notation and context, our filtered vector space analogue of \cite[Theorem 5.11]{Codogni:Patakfalvi:2021} is formulated in the following way.  It is an application of Theorem \ref{main:theorem:intro} and is proved in Section \ref{filt:HN:tensor:cor:Proof}.

\begin{theorem}\label{filt:HN:tensor:cor}
Suppose that a filtered vector space $V$ has positive Harder and Narasimhan data $\operatorname{HN}(\overrightarrow{\mu},\overrightarrow{r})$.  Fixing a nonnegative integer $z \geq 0$, for each nonnegative integer $m \geq 0$, let $H^{\# S_{m,z}(\overrightarrow{\mu})}$ be the subspace of $V^{\otimes m}$ that is given by \eqref{HN:subspace}.  Then, with this notation and hypothesis, it holds true that
$$
\lim_{m \to \infty} \frac{\dim H^{\# S_{m,z} (\overrightarrow{\mu})} }{ \dim V^{\otimes m} } = 1 \text{.}
$$
\end{theorem}

As some additional context and motivation for our abstract formulation of \cite[Theorem 5.11]{Codogni:Patakfalvi:2021}, see Theorems \ref{main:theorem:intro} and \ref{filt:HN:tensor:cor}, we mention that a concept of \emph{Harder and Narasimhan Polygons} emerged as a tool for expanding upon Harder and Narasimhan's theory of filtrations for vector bundles on curves. This was popularized by Grayson in his work on lattice reduction theory \cite{Grayson:1984}.  We refer to \cite{Casselman:2004} for an exposition.   Here, our point of departure is to associate a polygon to \emph{Harder and Narasimhan data}.   This is made precise in Section \ref{HN:polygons:and:vertices}.

Our results also give impetus for further investigation.  As some examples, the concept of Harder and Narasimhan data, see Definitions \ref{HN:defns}, raise the question of explicit and robust construction thereof, and especially with applications in geometric and arithmetic contexts.  

In the direction of  
Theorem \ref{main:theorem:intro}, there is the question of the rate in which the limit that appears in its conclusion is actually achieved.  Further, it remains interesting to understand the extent to which Theorems \ref{main:theorem:intro} and \ref{filt:HN:tensor:cor}, or variants thereof, will have applications for geometric and arithmetic aspects of filtered linear series.  Such results would complement those of our recent work \cite{Grieve:MVT:2019} (among others).

In Section \ref{Probability:theory}, we discuss the Laws of Large Numbers and state the Central Limit Theorem for independent and identically distributed random variables.  In Section \ref{Harder:Narasimhan:Main:Result}, we recall the classical theory of Harder and Narasimhan and state the main result from \cite{Harder:Narasimhan:1974}.  (See Theorem \ref{HN:Existence:Thm}.)  In Section \ref{tensor:product:construction}, we discuss a key motivational example from \cite{Codogni:Patakfalvi:2021}.  It also provides additional context and motivation for Theorem \ref{filt:HN:tensor:cor}.   Again, the concept of Harder and Narasimhan data is made precise in Section \ref{HN:polygons:and:vertices}.  Respectively, Theorems \ref{main:theorem:intro} and \ref{filt:HN:tensor:cor} are established in Sections \ref{Main:Theorem:Proof} and \ref{filt:HN:tensor:cor:Proof}.

Throughout this article, unless explicitly stated otherwise, all schemes, stacks and vector spaces are defined over a fixed algebraically closed characteristic zero base field $\kk$.  

\subsection*{Acknowledgements}

This work benefited from a visit to the American Institute of Mathematics during January 2020.  I thank Zsolt Patakfalvi for his lecture.  Moreover, I thank colleagues for conversations on related topics and for their insightful virtual Zoom lectures.  I also  thank an anonymous referee for their interest, careful reading of this work and for providing helpful thoughtful comments.
Finally, I thank the Natural Sciences and Engineering Research Council of Canada for their support through my grants DGECR-2021-00218 and RGPIN-2021-03821.

\section{Laws of Large Numbers}\label{Probability:theory}

In this section we state the three main Laws of Large Numbers.  Our approach mostly follows \cite{Billingsley:1995}, \cite{Durrett:2019} and \cite{Murty:2020}.

First of all, in heuristic terms, the \emph{Weak Law of Large Numbers} says that if $n$ is large, then there is only a small chance that the fraction of heads in $n$ tosses will be far from $1/2$.  In more precise terms, it is formulated in the following way.

\begin{theorem}[Weak Law of Large Numbers, cf. {\cite[p. 86]{Billingsley:1995}}, {\cite[Theorem 2.2.12]{Durrett:2019}}]\label{Weak:Law:Large:Numbers}
Let $X_1, X_2,\dots$ be a sequence of independent identically distributed simple random variables with identical expected values 
$$E(X_n) = \mu\text{,}$$ for each $n$.  Put 
$$S_n = X_1 + \hdots + X_n\text{.}$$  Suppose that $y > 0$.  Then
$$
\lim_{n \to \infty} \operatorname{Prob}\left( \left| n^{-1} S_n - \mu \right| \geq y  \right) = 0 \text{.}
$$
\end{theorem}

Turning to the \emph{Strong Law of Large Numbers}, recall that it expands upon the Weak Law (Theorem \ref{Weak:Law:Large:Numbers}).

\begin{theorem}[Strong Law of Large Numbers, cf. {\cite[Theorem 6.1]{Billingsley:1995}}, {\cite[Theorem 2.4.1]{Durrett:2019}}]
Let $X_1, X_2,\dots$ be a sequence of independent identically distributed simple random variables with identical expected values 
$$E(X_n) = \mu\text{,}$$ 
for each $n$.  Put 
$$S_n = X_1 + \hdots + X_n\text{.}$$  
It then holds true that
$$
\operatorname{Prob}\left(\lim_{n \to \infty} n^{-1} S_n = \mu \right) = 1 \text{.}
$$
\end{theorem}

Finally, we state  the Central Limit Theorem.  It also extends the Weak Law of Large Numbers.  An attractive self contained proof via the theory of Fourier transforms is given in \cite[Section 3.15]{Murty:2020}.

\begin{theorem}[Central Limit Theorem, cf. {\cite[Theorem 27.1]{Billingsley:1995}}, {\cite[
Theorem 3.4.1]{Durrett:2019}}, {\cite[Theorem 3.16]{Murty:2020}}]\label{Central:Limit:Theorem}
Let $\Phi$ be a standard normal random variable and so, in particular, having probability density function
$$
\phi(y) = \frac{1}{\sqrt{2 \pi}} e^{- \frac{y^2}{2} } \text{.}
$$
Let $X_1, X_2,\dots$ be a sequence of independent random variables having the same distribution with mean $\mu$ and positive finite variance $\sigma^2$.  Let 
$$S_n = X_1 + \hdots + X_n\text{.}$$  
It then holds true that
$$
\lim_{n \to \infty} \frac{S_n - n \mu}{\sigma \sqrt{n}} \to \Phi \text{.}
$$
Further
$$
\lim_{n \to \infty} n^{-1} S_n \to \mu \text{.}
$$
\end{theorem}

\section{The theory of Harder and Narasimhan}\label{Harder:Narasimhan:Main:Result}

Here, we recall, from \cite{Harder:Narasimhan:1974}, the theory of Harder and Narasimhan.
In what follows, $C$ denotes a non-singular projective algebraic curve over an  arbitrary characteristic algebraically closed base field.  By abuse of terminology, we fail to distinguish amongst the concept of finite rank locally free $\Osh_C$-modules and total spaces of vector bundles.  On the other hand, if $\mathcal{E}$ is a vector bundle on $C$ and $\mathcal{F}$ a locally free submodule, then $\mathcal{F}$ is called a subvector bundle if the quotient $\mathcal{E} / \mathcal{F}$ is locally free.

If $\mathcal{E}$ is a nonzero vector bundle on $C$, then its \emph{slope} is
$$
\mu(\mathcal{E}) := \frac{\operatorname{deg}\left(\mathcal{E} \right)}{\operatorname{rank}\left(\mathcal{E}\right)} = \frac{ \operatorname{deg}\left( \operatorname{det}\left(\mathcal{E}\right)\right)}{r} \text{.}
$$
Here 
$$r := \operatorname{rank}\left(\mathcal{E}\right)$$ 
is the \emph{rank} of $\mathcal{E}$ and 
$$
\operatorname{det}\left(\mathcal{E}\right) := \bigwedge^r\left(\mathcal{E}\right)
$$ 
is the \emph{determinant line bundle}.

If a nonzero vector bundle $\mathcal{E}$ on $C$ has the property that
$$
\mu(\mathcal{F}) \leq \mu(\mathcal{E}) \text{,}
$$
for all nonzero proper subbundles $\mathcal{F}$, then it is called \emph{semistable}.  By \cite[Lemma 1.3.7]{Harder:Narasimhan:1974}, if a nonzero vector bundle $\mathcal{E}$ is not semistable, then it admits a unique non-zero subbundle $\mathcal{F}$ which \emph{stongly contradicts semistability} in the following sense
\begin{enumerate}
\item[(i)]{$\mathcal{F}$ is semi-stable; and}
\item[(ii)]{if $\mathcal{G}$ is a subbundle of $\mathcal{E}$ which contains $\mathcal{F}$ as a subbundle, then 
$$\mu(\mathcal{F}) > \mu(\mathcal{G})\text{.}$$}
\end{enumerate}

The concept of subbundles which strongly contradict semistability is a key technical point in establishing the existence and uniqueness of the Harder and Narasimhan (canonical) filtrations.  This fundamental result of \cite{Harder:Narasimhan:1974} is formulated in the following way.

\begin{theorem}[{\cite[Section 1.3]{Harder:Narasimhan:1974}}] \label{HN:Existence:Thm} 
 Let $C$ denote a non-singular projective algebraic curve over an arbitrary characteristic algebraically closed base field.
Let $\mathcal{E}$ be a nonzero vector bundle on $C$.  Then $\mathcal{E}$ admits a uniquely determined flag of subvector bundles
\begin{equation}\label{HN:eqn1}
0 = \mathcal{E}_0 \subsetneq \mathcal{E}_1 \subsetneq \hdots \subsetneq \mathcal{E}_\ell = \mathcal{E} 
\end{equation}
which has the two properties that
\begin{enumerate}
\item[(i)]{
the successive quotients $\mathcal{E}_i / \mathcal{E}_{i-1}$ are semistable for $i = 1,\dots,\ell$; and
}
\item[(ii)]{
the slopes of the successive quotients are strictly decreasing in the sense that
$$
\mu(\mathcal{E}_i / \mathcal{E}_{i-1}) > \mu(\mathcal{E}_{i+1} / \mathcal{E}_i)
$$
for $i = 1,\dots,\ell-1$.
}
\end{enumerate}
\end{theorem}

In working with the Harder and Narasimhan filtrations \eqref{HN:eqn1}, in what follows we find it useful to put
%$$
%r := \operatorname{rank}(\mathcal{E}) \text{, }
%$$ 
$$
\mu_i := \mu(\mathcal{E}_i / \mathcal{E}_{i-1})
%$$
\text{ and }
%$$
r_i := \operatorname{rank}(\mathcal{E}_i / \mathcal{E}_{i-1}) \text{,}
$$
for $i = 1,\dots,\ell$.
By this notation 
$r_i > 0\text{,}$ 
for all $i = 1,\dots,\ell$,
$$
\mu_1 > \dots > \mu_\ell 
$$
and
$$
\sum_{i=1}^\ell r_i = r \text{.}
$$
Moreover, if
$$
\mu := \mu\left(\mathcal{E}\right) \text{, }
$$
then
$$
\mu = \frac{ \sum_{i=1}^\ell \mu_i r_i }{r} \text{.}
$$
Finally, we refer to the bundles $\mathcal{E}_i$, for $i = 1,\dots,\ell$, which arise in the filtration \eqref{HN:eqn1}, as the \emph{Harder and Narasimhan subbundles} of $\mathcal{E}$. 

\section{The Harder and Narasimhan polygons and their vertices}\label{HN:polygons:and:vertices}

A concept of \emph{Harder and Narasimhan Polygons} emerged as a tool for expanding upon Harder and Narasimhan's theory of filtrations for vector bundles on curves. This was popularized by Grayson in his work on lattice reduction theory \cite{Grayson:1984}.  We refer to the article \cite{Casselman:2004}, by Casselman, for an exposition.   Here, our viewpoint is to associate a polygon to \emph{Harder and Narasimhan data}.  We make this precise in Definitions \ref{HN:defns}. 

Another context in which a fruitful theory of Narasimhan Polygons has emerged is that of filtered vector spaces.  Such developments have been made possible by work of Faltings and W\"{u}stholz \cite{Faltings:Wustholz}, Chen \cite{Chen:2010:b} and others.  In Example \ref{filt:v:sp:example}, below, and because of its relevance to Theorem \ref{filt:HN:tensor:cor}, we indicate how our notion of \emph{Harder and Narasimhan data} fits within the framework of the Harder and Narasimhan filtration that is associated to each filtered vector space.  

Within the context of Diophantine approximation, the theory of Harder and Narasimham filtrations for vector spaces is an important aspect of the work of Faltings and W\"{u}stholz \cite{Faltings:Wustholz}.  Similar filtrations together with a theory of polygons arise, more recently, in the work of  Evertse and Ferretti (see \cite[Section 15]{Evertse:Ferretti:2013}).  

Of interest, for our purposes here, is the \emph{asymptotic probabilistic nature} of such \emph{Harder and Narasimhan Polygons}.  Exactly what is meant by this is  made precise below.

\begin{defns}\label{HN:defns}
By \emph{Harder and Narasimhan data}, is meant a collection of strictly decreasing rational numbers 
$$\mu_1 > \dots > \mu_\ell$$ 
together with a collection of positive  integers $r_i > 0$, for $i = 1,\dots,\ell$.  Denote such data by $\operatorname{HN}\left(\overrightarrow{\mu},\overrightarrow{r}\right)$.  Here 
$$\overrightarrow{\mu} := (\mu_1,\dots,\mu_\ell) 
\text{ and }
\overrightarrow{r} := (r_1,\dots,r_\ell)\text{.}$$  
The integer $\ell$ is called the \emph{length}.  

Put 
$$r := \sum_{i=1}^\ell r_i\text{.}$$  
This is the \emph{rank} and the $r_i$ are the \emph{subquotient ranks}.  

Finally, upon setting 
$$d_i := r_i \mu_i\text{,}$$ 
for $i = 1,\dots,\ell$, and 
$$\overrightarrow{d} := (d_1,\dots,d_\ell)\text{,}$$ 
we obtain the \emph{degree vector}. If 
$$\left| \overrightarrow{d} \right| := \sum_{i=1}^\ell d_i > 0 \text{,}$$ 
then the Harder and Narasimhan data $\operatorname{HN}\left(\overrightarrow{\mu},\overrightarrow{r}\right)$ is said to have \emph{positive degree}.  

Further, set
$$
\overrightarrow{p} := (p_1,\dots,p_\ell)
$$
where
$$
p_i := \frac{r_i}{r} \text{,}
$$
for $i = 1,\dots,\ell$.  Then the $p_i$ are the Harder and Narasimhan \emph{probabilities} while $\overrightarrow{p}$ is the Harder and Narasimhan \emph{probability vector}.

Finally, similar to the approach of \cite[Definition 1.10 and Discussion 1.16]{Grayson:1984} and \cite[p.~630]{Casselman:2004}, the \emph{vertices} of Harder and Narasimhan data $\operatorname{HN}\left(\overrightarrow{\mu},\overrightarrow{r} \right)$,  is defined to be the collection of points
$$
\operatorname{Verticies}\left(\overrightarrow{r}, \overrightarrow{d}\right) := \left\{ (r_i, d_i) : i = 1,\dots,\ell \right\} \text{.}
$$
(Compare also with \cite[Remark 2.2.8]{Chen:2010:b}.)
\end{defns}

\begin{example}\label{filt:v:sp:example}  As in \cite[Section 4]{Faltings:Wustholz},
Let $V$ be a finite dimensional $\kk$-vector space and consider, given a collection of real numbers
$$
0 \leq \lambda_0 < \lambda_1 < \dots < \lambda_n < \lambda_{n+1}\text{,}
$$
a filtration of the form
\begin{equation}\label{filt:vsp:eqn1}
V = F^{\lambda_0} V \supsetneq F^{\lambda_1} V \supsetneq \dots \supsetneq F^{\lambda_n} V \supsetneq F^{\lambda_{n +1}}V = 0 \text{.}
\end{equation}
The filtered vector space \eqref{filt:vsp:eqn1} is called \emph{semistable} if it holds true that 
\begin{equation}\label{filt:vsp:eqn2}
\mu(V') \leq \mu(V)
\end{equation}
for each proper nonzero subspace
$$
0 \not = V' \subsetneq V \text{.}
$$
In \eqref{filt:vsp:eqn2}, the slope of $V'$ is calculated with respect to the induced filtration.  

As noted by Faltings and W\"{u}stholz, each such filtered vector space \eqref{filt:vsp:eqn1} admits a \emph{canonical Harder and Narasimhan filtration}.  
In more specific terms, there exists a flag of vector spaces
$$
0 = V_0 \subsetneq V_1 \subsetneq \dots \subsetneq V_{\ell} = V \text{,}
$$
which have the property that each successive quotient $V_i / V_{i - 1}$, for $i = 1,\dots,\ell$, is semistable (which respect to the induced filtration), and furthermore, the slopes $\mu(V_i / V_{i-1})$ of the successive quotients are strictly decreasing.

Within this context, putting 
$$
\mu_i := \mu(V_i / V_{i-1}) \text{
and }
r_i := \dim(V_i / V_{i-1}) \text{,}
$$
for $i = 1,\dots,\ell$ and setting
$$
\overrightarrow{\mu} := (\mu_1,\dots,\mu_\ell)
\text{ 
and
 }
\overrightarrow{r} := (r_1,\dots,r_\ell)
$$
yields Harder and Narasimhan data $\operatorname{HN}(\overrightarrow{\mu},\overrightarrow{r})$.
\end{example}

\section{Tensor products of Harder and Narasimhan subbundles}\label{tensor:product:construction}

To provide motivation for our construction with vertices  of Harder and Narasimhan subbundles in Section \ref{Main:Theorem:Proof}, here we discuss a related construction of \cite{Codogni:Patakfalvi:2021}, which was the starting point for our investigation here.  

These constructions from \cite{Codogni:Patakfalvi:2021}, build on Viehweg's fiber product approach from \cite{Viehweg:1983}, for proving weak positivity results for direct images of tensor powers of relative canonical sheaves.

A representative example for the techniques of \cite[Section 5]{Codogni:Patakfalvi:2021}, is explained in \cite[Remark 5.6]{Codogni:Patakfalvi:2021}.  We reproduce some of that discussion here.   It provides motivation for our analogous results which apply to the context of filtered vector spaces. (See Theorem \ref{filt:HN:tensor:cor} and Example \ref{filt:v:sp:example}.)

In particular, working over the projective line $\PP^1_{\kk}$, consider the case of a vector bundle
$$
\mathcal{E} := \bigoplus_{1 \leq i \leq \ell} \Osh_{\PP^1_{\kk}}(b_i)^{\oplus n_i}
$$
where 
$b_i \in \ZZ$ 
and 
$b_i < b_{i+1}\text{,}$ 
for $i = 1,\dots,\ell-1$.  

With respect to the Harder and Narasimhan filtrations for $\mathcal{E}$, the $i$th Harder and Narasimhan submodule is
$$
\mathcal{E}_i := \bigoplus_{1 \leq j \leq i} \Osh_{\PP^1}(b_j)^{\oplus n_j} \text{.}
$$
The $i$th subquotient slope is thus 
$$
\mu_i := \mu\left( \mathcal{E}_i / \mathcal{E}_{i-1} \right) = b_i \text{.}
$$
Now, fixing a positive integer $m > 0$, the idea is to study via probabilistic methods, inside of $\mathcal{E}^{\otimes m}$, the prevalence of those submodules which are obtained via the $m$-fold tensor products
$$
\bigotimes_{i=1}^m \mathcal{E}_{a_i} \subseteq \mathcal{E}^{\otimes m} \text{.}
$$
Here, 
$a_i \in \{1,\dots,\ell \}\text{,}$
for $i = 1,\dots,m$, and 
$$
\sum_{i=1}^m \mu_{a_i} \geq z \text{,}
$$
for 
$z \in \ZZ_{\geq 0}$ 
some given nonnegative integer.

In Section \ref{Main:Theorem:Proof}, we generalize this construction from \cite{Codogni:Patakfalvi:2021}, so as to treat the more general context of Harder and Narasimhan vertices.
 
\section{Statement of Main Theorem and its proof}\label{Main:Theorem:Proof}

In order to state our main theorem (see Theorem \ref{HN:vertex:CLT} below), let us first consider the construction from Section \ref{tensor:product:construction} in a slightly more general context.

Let 
$[\ell]:= \{1,\dots,\ell\}\text{.}$  
Fixing a positive integer $m \in \ZZ_{>0}$, let $[\ell]^m$ be the $m$-fold Cartesian product of $[\ell]$.  

Let $\succ$ be the partial order on $[\ell]^m$ that is defined in the following way.  If 
$\mathbf{a}, \mathbf{b} \in [\ell]^m\text{,}$ 
then 
$\mathbf{a} \succeq \mathbf{b}$ 
if 
$a_j \geq b_j$ 
for all $j = 1,\dots,m$; whereas 
$\mathbf{a} \succ \mathbf{b}$ 
if 
$\mathbf{a} \succeq \mathbf{b}$ 
and 
$a_j > b_j$ 
for some $j = 1,\dots, m$.

Now, fix Harder and Narasimhan data $\operatorname{HN}\left(\overrightarrow{\mu},\overrightarrow{r}\right)$.   Of particular interest, is the case that the Harder and Narasimhan data $\operatorname{HN}\left(\overrightarrow{\mu},\overrightarrow{r}\right)$ has positive degree.  But, on the other hand, we do not impose that condition for the present time being.

Inside of $[\ell]^m$, define the subset $S_{m,z}(\overrightarrow{\mu})$, for 
$z \in \ZZ_{\geq 0}$ 
a fixed nonnegative integer, by the condition that
$$
S_{m,z}\left(\overrightarrow{\mu}\right) := \left\{ \mathbf{a} = (a_1,\dots,a_m) \in [\ell]^m : \sum_{j=1}^m \mu_{a_j} \geq z \right\} \text{.}
$$
If 
$
\mathbf{a} \in [\ell]^m \text{,}
$
then put
$$
v_{\overrightarrow{\mu}}(\mathbf{a}) := \sum_{j=1}^m \mu_{a_j} \text{.}
$$

The set $S_{m,z}\left(\overrightarrow{\mu}\right)$ has the following two properties:
\begin{enumerate}
\item[(i)]{
if $\mathbf{a} \in S_{m,z}(\overrightarrow{\mu})\text{,}$ 
$\mathbf{b} \in [\ell]^m$ 
and 
$\mathbf{a} \succeq \mathbf{b}\text{,}$ 
then 
$\mathbf{b} \in S_{m,z}(\overrightarrow{\mu})\text{;}$ 
and
}
\item[(ii)]{
if 
$\mathbf{a} \in [\ell]^m\text{,}$ 
then 
$v_{\overrightarrow{\mu}}(\mathbf{a}) \geq z\text{,}$ 
if and only if 
$\mathbf{a} \in S_{m,z}(\overrightarrow{\mu})\text{.}$
}
\end{enumerate}

Henceforth put
$$
d := \# S_{m,z}(\overrightarrow{\mu}) \text{.}
$$
Then, upon arranging the elements of $[\ell]^m$ in decreasing order, with respect to $v_{\overrightarrow{\mu}}(\cdot)$, we may write
$$
[\ell]^m = S_{m,z}(\overrightarrow{\mu}) \bigsqcup \{\mathbf{a}_{d+1},\dots,\mathbf{a}_e \}
$$
where
$$
S_{m,z}(\overrightarrow{\mu}) = \{\mathbf{a}_1,\dots,\mathbf{a}_d\} \text{,}
$$
$$
v_{\overrightarrow{\mu}}(\mathbf{a}_{d+1}) \geq \dots \geq v_{\overrightarrow{\mu}}(\mathbf{a}_e)
$$
and
$
e = \ell m \text{.}
$
In this way, we may denote elements of $[\ell]^m$ as 
$$
\mathbf{a}_i := (a_{1i},\dots,a_{mi}) \text{,}
$$
for $1 \leq i \leq e$.

Recall, the Harder and Narasimhan probabilities
$$
p_i := \frac{r_i}{r} \text{,}
$$
for $i = 1,\dots,\ell$.  Then
$$
\sum_{i=1}^\ell p_i = 1
$$
and, for later use in Section \ref{filt:HN:tensor:cor:Proof}, we note that
\begin{equation}\label{CLT:prob:sum:random:var}
%\begin{split}
\sum_{i = 1}^d \prod_{j=1}^m \frac{r_{a_{ji}}}{r}  = \sum_{\mathbf{a} \in S_{m,z}(\overrightarrow{\mu})} \prod_{j=1}^m \frac{r_{a_j}}{r} \\ 
 = \sum_{\mathbf{a} \in S_{m,z}(\overrightarrow{\mu})} \prod_{j=1}^m p_{a_j} \text{.}
%\end{split}
%\end{align}
\end{equation}

Having fixed notation as above, our main theorem (Theorem \ref{HN:vertex:CLT} below) expands upon \cite[Theorem 5.11]{Codogni:Patakfalvi:2021}.

\begin{theorem}\label{HN:vertex:CLT}
Let $\operatorname{HN}\left(\overrightarrow{\mu}, \overrightarrow{r} \right)$ be Harder and Narasimhan data.  Assume that $\operatorname{HN}\left(\overrightarrow{\mu},\overrightarrow{r}\right)$ has positive degree.  Let $\overrightarrow{p} = (p_1,\dots,p_\ell)$ be its probability vector.  Fix a positive integer $m>0$ and a nonnegative integer $z \geq 0$.  It then holds true that
$$
\lim_{m \to \infty} \sum_{\mathbf{a} \in S_{m,z}(\overrightarrow{\mu}) } \prod_{j=1}^m p_{a_j} = 1 \text{.}
$$
\end{theorem}

\begin{proof}[Proof of Theorems \ref{HN:vertex:CLT} and \ref{main:theorem:intro}]
Let 
$$\mathcal{Y} = \mathcal{Y}\left([\ell],p_1,\dots,p_{\ell}\right)$$ 
be the discrete probability  space on the set $[\ell]$ and having probability measures $p_1,\dots,p_\ell$.  Let 
$$Y_j := Y_j(\overrightarrow{\mu})$$ 
be a sequence of independent, identically distributed random variables of $\mathcal{Y}$ that take value $\mu_i$ on $i$.  

Let 
$$
Z_m := \sum_{j=1}^m Y_j \text{.}
$$
Then, by considering the definition of the set $S_{m,z}(\overrightarrow{\mu})$ and the random variables $Y_1, Y_2,\dots$, it follows that 
\begin{equation}\label{CLT:Main:Eqn2}
%\begin{split}
 \operatorname{Prob}\left( Z_m \geq z \right)  = \operatorname{Prob}\left(\sum_{j=1}^m Y_j \geq z \right) 
  =
  \sum_{\mathbf{a} \in S_{m,z}(\overrightarrow{\mu})} \prod_{j=1}^m p_{a_j} \text{.}
 % \end{split}
\end{equation}

Now, we apply the Central Limit Theorem (Theorem \ref{Central:Limit:Theorem}), to show that 
\begin{equation}\label{CLT:Main:Eqn3}
\lim_{m \to \infty} \operatorname{Prob}\left( Z_m \geq z \right) = 1 \text{.}
\end{equation}
To this end, first note that $\mathcal{Y}$ is a finite metric space.  Moreover, the independent and identically distributed random variables $Y_j$ have finite mean and variance.  These quantities are independent of $j$.  Denote them, respectively, by $\mu$ and $\sigma$.

The Central Limit Theorem, for independent identically distributed random variables (Theorem \ref{Central:Limit:Theorem}), thus implies that the random variable 
$$\frac{Z_m - m \mu}{\sqrt{m}}$$ 
converges weakly to a normal distribution $\Phi$ with expected value $0$ and covariance $\sigma^2$.

In particular, for each real number $y$ it holds true that
$$
\lim_{m\to \infty} \operatorname{Prob}\left( \frac{ Z_m - m \mu}{ \sqrt{m} } \geq y \right) = \operatorname{Prob}\left( \Phi \geq y \right) \text{.}
$$

Now, recall that the Harder and Narasimhan data $\operatorname{HN}\left(\overrightarrow{\mu},\overrightarrow{r}\right)$ has positive degree. It thus follows that
$$
\mu = \sum_{i=1}^\ell \mu_i p_i = \frac{\sum_{i=1}^\ell \mu_i r_i }{r} > 0 \text{.}
$$
Thus, fixing a real number $y$, there is a positive integer 
$m_y > 0$ 
such that if 
$m \geq m_y\text{,}$ 
then
$$
z \leq y \sqrt{m} + m \mu \leq Z_m
$$
provided that
$$
\frac{Z_m - m \mu}{\sqrt{m}} \geq y \text{.}
$$

The above discussion implies, in particular, that
$$
\liminf_{m \to \infty} \operatorname{Prob} \left(Z_m \geq z \right) \geq \operatorname{Prob} \left(\Phi \geq y \right)
$$
for all real numbers $y$.
On the other hand, note that
$$
\lim_{y \to - \infty} \operatorname{Prob} \left(\Phi \geq y \right) = 1 \text{.}
$$
Finally, since
$$
\liminf_{m \to \infty} \operatorname{Prob}\left( Z_m \geq z \right) = 1
$$
and
$$
\operatorname{Prob}\left( Z_m \geq z \right) \leq 1 \text{,}
$$
for all $m$, it holds true that
$$
\lim_{m \to \infty} \operatorname{Prob}\left( Z_m \geq z \right) = 1 \text{.}
$$
This establishes the validity of the relation \eqref{CLT:Main:Eqn3}.
\end{proof}

\begin{remark}
As mentioned in \cite[Remark 5.12]{Codogni:Patakfalvi:2021}, the relation  \eqref{CLT:Main:Eqn3} may be established using the Chebyshev's inequality in place of the Central Limit Theorem.
\end{remark}

\section{Proof of Theorem \ref{filt:HN:tensor:cor}}\label{filt:HN:tensor:cor:Proof}

Finally, we indicate the manner in which Theorem \ref{filt:HN:tensor:cor} follows from Theorems \ref{HN:vertex:CLT} and \ref{main:theorem:intro}.  The key point to the proof is to adapt the proof of \cite[Proposition 5.9]{Codogni:Patakfalvi:2021} to the context of filtered vector spaces.  Having done this, Theorem \ref{filt:HN:tensor:cor} follows from Theorem \ref{HN:vertex:CLT}, essentially because of the relation \eqref{CLT:prob:sum:random:var}.

\begin{proof}[Proof of  Theorem \ref{filt:HN:tensor:cor}]

Consider the Harder and Narasimhan filtration \eqref{filt:vp:space:eqn2}
that is associated to the filtered vector space \eqref{filt:vp:space:eqn1}.
Recall, respectively, the corresponding slopes, degrees and ranks
$$
\mu_i = \mu(V_i / V_{i-1}) \text{, }
r_i = \dim (V_i / V_{i-1}) \text{ and }
d_i = r_i \mu_i \text{,}
$$
for $i = 1,\dots,\ell$.  
Put 
$$
r := \dim V \text{.}
$$

For each
$$
\mathbf{a} \in [\ell]^m \text{,}
$$
recall that we have put
$$
v_{\overrightarrow{\mu}}(\mathbf{a}) = \sum_{j=1}^m \mu_{a_j} \text{.}
$$
Moreover, we arrange the elements of $[\ell]^m$ is decreasing order with respect to $v_{\overrightarrow{mu}}(\cdot)$.  We write this as
$$
[\ell]^m := \{\mathbf{a}_1=(a_{11},\dots,a_{m1}),\dots,\mathbf{a}_{\ell m} = (a_{1 \ell m },\dots,a_{m \ell m}) \} \text{.}
$$
Finally, we have defined the set
$$
S_{m,z}(\overrightarrow{\mu}) := \left\{\mathbf{a} = (a_1,\dots,a_m) \in [\ell]^m : \sum_{j=1}^m \mu_{a_j} \geq z \right\} \text{;}
$$
then $S_{m,z}(\overrightarrow{\mu})$ consists of the the first $\# S_{m,z}(\overrightarrow{\mu})$ elements of $[\ell]^m$, with respect to the ordering that is induced by $v_{\overrightarrow{mu}}(\cdot)$.

Now, for all $1 \leq i \leq \ell m$, inside of $V^{\otimes m}$, define
$$
F^i := \bigotimes_{j=1}^m V_{a_{ji}} \text{, }
$$
$$
H^i := \sum_{j=1}^i F^i 
$$
and
$$
G^i := \bigotimes_{j=1}^m \left(F^{a_{ji}} / F^{a_{ji} - 1} \right) \text{.}
$$
Then
$$
\dim(G^i) = \prod_{j=1}^m r_{a_{ji}}\text{.}
$$
Let us also mention that, with respect to the filtration on $G^i$ that is induced by the given filtration on $V$, the slope of $G^i$ is 
$$
\mu(G^i)  = \sum_{j=1}^m \mu(F^{a_{ji}} / F^{a_{ji}-1}) 
 = \sum_{j=1}^m \mu_{a_{ji}} \text{.}
$$
(We refer to \cite[Section 1]{Fujimori:2003}, for example, for more details about the behaviour of slopes under taking tensor products and quotients.)

We next 
establish isomorphisms
\begin{equation}\label{isom:eqn:1}
H^i / H^{i-1} \simeq G^i \text{,}
\end{equation}
for $1 \leq i \leq \ell m$.  The isomorphisms \eqref{isom:eqn:1} are a key point to the proof of Theorem \ref{filt:HN:tensor:cor}.  
Assuming their existence it then follows that 
\begin{equation}\label{isom:eqn:2}
\dim H^{\# S_{m,z}(\overrightarrow{\mu})}  = \sum_{i=1}^{\# S_{m,z}(\overrightarrow{\mu})} \dim G^i \\
 = \sum_{i=1}^{\# S_{m,z}(\overrightarrow{\mu})} \prod_{j=1}^m r_{a_{ji}} \text{.}
\end{equation}
The desired conclusion then follows, since, in light of \eqref{isom:eqn:2}, it follows that
$$
\frac{\dim H^{\# S_{m,z}(\overrightarrow{\mu})}}{\dim V^{\otimes m}} = \sum_{i=1}^{\# S_{m,z}(\overrightarrow{\mu})} \prod_{j=1}^m \frac{r_{a_{ji}}}{r} \text{.}
$$

But, by \eqref{CLT:prob:sum:random:var}, this may be rewritten as
$$
\frac{\dim H^{\# S_{m,z}(\overrightarrow{\mu})}}{\dim V^{\otimes m}}  = \sum_{ \mathbf{a} \in S_{m,z}(\overrightarrow{\mu})} \prod_{j=1}^m \frac{r_{a_j}}{r} \\
 = \sum_{ \mathbf{a} \in S_{m,z}(\overrightarrow{\mu})} \prod_{j=1}^m p_{a_j} \text{.}
$$
In particular, it then follows from Theorem \ref{HN:vertex:CLT} that 
$$
\lim_{m \to \infty} \frac{\dim H^{\# S_{m,z} (\overrightarrow{\mu})}}{ \dim V^{\otimes m} }  = \lim_{m\to\infty} \sum_{\mathbf{a} \in S_{m,z}(\overrightarrow{\mu})} \prod_{j=1}^m p_{a_j} = 1 \text{.}
$$

It remains to establish the isomorphisms \eqref{isom:eqn:1}, for each $1 \leq i \leq \ell m$.  To this end, we first make note the following vector space isomorphisms
$$
H^{i} / H^{i-1} = (H^{i-1} + F^i) / H^{i-1} \simeq F^i / \left(F^i \bigcap H^{i-1}\right)
$$
and
$$
G^i \simeq F^i / \left(\sum_{\mathbf{a}_{i} \succ \mathbf{a}_{i'}} F^{i'}\right) \text{;}
$$
there is a naturally defined surjection
\begin{equation}\label{isom:eqn:3}
G^i \twoheadrightarrow H^i / H^{i-1} \text{.}
\end{equation}

But, on the other hand
$$
\dim(V^{\otimes m}) = \sum_{i=1}^{\ell m} \dim (H^i / H^{i-1}) \leq \sum_{i=1}^{\ell m} \dim G^i
$$
and
$$
\sum_{i=1}^{\ell m} \dim G^i  = \sum_{\mathbf{a}_i = (a_{1i},\dots,a_{mi}) \in [\ell]^m} \left( \prod_{j=1}^m r_{a_{ji}}\right) 
= \left(\sum_{i=1}^\ell r_i \right)^m 
 = \dim V^{\otimes m} \text{.}
$$
The conclusion is then that 
$$
\dim(V^{\otimes m}) = \sum_{i =1}^{\ell m} \dim G^i  \text{,}
$$
whence, by dimension reasons,
$$
\dim(H^i / H^{i-1}) = \dim G^i \text{,}
$$
for all $0 \leq i \leq \ell m$.  The surjections \eqref{isom:eqn:3} and thus injective. The desired isomorphisms \eqref{isom:eqn:1} are thus established.
\end{proof}

\providecommand{\bysame}{\leavevmode\hbox to3em{\hrulefill}\thinspace}
\providecommand{\MR}{\relax\ifhmode\unskip\space\fi MR }
% \MRhref is called by the amsart/book/proc definition of \MR.
\providecommand{\MRhref}[2]{%
  \href{http://www.ams.org/mathscinet-getitem?mr=#1}{#2}
}
\providecommand{\href}[2]{#2}

\end{document}